\documentclass[12pt,a4,reqno]{amsart}
\usepackage{amsmath, amssymb}
\usepackage{graphicx}
\usepackage{subfigure}
 \usepackage{mathrsfs}
 \usepackage{bbm}
 \usepackage{subcaption}
\usepackage{color}
\usepackage{xpatch} 
\usepackage{esint}
\usepackage[svgnames]{xcolor} 
\usepackage[colorlinks,citecolor=red,pagebackref,hypertexnames=false,breaklinks]{hyperref}
\usepackage{pgf,tikz}
\usepackage{pdfsync}
\usepackage{comment}
\usepackage[bottom]{footmisc}
\usepackage{dsfont}
\usepackage{amssymb}
\usepackage{amsthm}
\newtheorem{theorem}{Theorem}

\newtheorem{definition}[theorem]{Definition}
\newtheorem{lemma}[theorem]{Lemma}
\newtheorem{corollary}[theorem]{Corollary}

\DeclareMathOperator*{\supp}{supp}

\def\Re{\mathbb{R}}

\def\Ce{\mathbb{C}}

\def\Ne{\mathbb{N}}

\def\l{\left}
\def\r{\right}

\def\L2{L^2(\mathbb{R})}
\def\t{\text}
\def\arg{\text{arg}}
\def\Real{\text{Re}}
\def\Imag{\text{Im}}
\def\for{\text{for}}
\def\e{\text{e}}
\def\i{\text{i}}
\def\d{\text{d}}

\xapptocmd{\proof}{\mbox{}\par\nobreak}{}{}
\begin{document}

\title[Stability estimated for the truncated Fourier transform]{Estimates on the stability constant for the truncated Fourier transform}

\author{Mirza Karamehmedovi\'c$^\star$}
\thanks{$\star$ Department of Applied Mathematics, Technical University of Denmark, Kgs. Lyngby, Denmark (mika@dtu.dk).}

\author{Martin Sæbye Carøe$^\dagger$}

\thanks{$\dagger$ Department of Applied Mathematics, Technical University of Denmark, Kgs. Lyngby, Denmark (s194317@student.dtu.dk).}

\author{Faouzi Triki$^\ddag$}

\thanks{$\ddag$ Laboratoire Jean Kuntzmann, Université Grenoble-Alpes, Grenoble, France (Faouzi.Triki@univ-grenoble-alpes.fr).}






\maketitle

\begin{abstract} In this paper we are interested in the inverse problem of recovering a compact supported function from
its truncated Fourier transform.  We derive new Lipschitz stability estimates for the inversion in terms of the truncation 
parameter. The obtained results show that the Lipschitz constant is of order one when the truncation parameter  is 
larger than the spatial frequency of the function, and it grows exponentially when the  truncation parameter tends to zero. 
Finally,  we present some numerical examples of reconstruction of a compactly supported function from its noisy truncated Fourier transform.
The numerical illustrations  validate our theoretical results.

\end{abstract}
\vspace{0.5cm}
{ Keywords:} Truncated Fourier transform; Stability estimates; Harmonic 
measure

\section{Introduction and  main results}
Let $F$ be the Fourier transform on the space of real valued square integrable functions on the interval $[-1,1]$, $L^2(-1,1)$.
\begin{equation*}
    (Ff)(\xi) = \widehat{f}(\xi) = \frac{1}{\sqrt{2\pi}} \int_{-1}^1 f(x)\e^{-\i x\xi }\d x, \quad \xi \in \Re.
\end{equation*}

Consider the truncated Fourier transform $F_Bf(\xi) = \chi_{[-B,B]}(\xi)  \widehat{f}(\xi)$, where $B>0$ and $\chi_A$ is the characteristic function of the subset $A\subseteq \Re$

\begin{equation*}
   \chi_A(\xi) = \begin{cases}
    1 & \xi \in A \\
    0 &  \xi \in \Re\setminus A.
\end{cases}
\end{equation*}

We will consider the stability of the inverse problem of recovering $f$ from the knowledge of $F_Bf$. The problem has many applications
in many scientific areas such as medical imaging,  radar imaging, and  is related to inverse scattering  \cite{bao2010multi, bao2011numerical, marechal2024regularization, natterer2001mathematics}. 

It is well known, see \cite[Section 7.1.14]{hörmander} \cite{bao2010multi}, that the Fourier transform of a compactly supported function, such as $f\in L^2(-1,1)$, has a holomorphic extension to $\Ce$. It follows that if $\widehat{f}(\xi) = 0$ for all $\xi \in [-B,B]$, then $\widehat{f}(\xi) = 0$ for all $\xi \in \Ce$. Thus $F_Bf\equiv 0$ implies that $\widehat{f}\equiv 0$ and hence $f \equiv 0$, which shows that the problem of recovering $f$ from $F_Bf$ has a unique solution.

The problem of recovering $f$ is however not stable as the singular values of $F_B$ decay exponentially, see \cite[Theorem 3.17]{osipov} \cite{bao2010multi}.

In the paper, \cite{zbMATH06805210}, it is shown that for a function $f\in H^1(-1,1)$, we have that for any $B>0$, there exists constants $c_B, c_B' >0$ (that may depend on $B$) such that with
\begin{equation*}
    k_{B}(f)= c_B' \|f_x\|_{L^2(-1,1)}/\|f\|_{L^2(-1,1)},
\end{equation*}
we have
\begin{equation*}
    c_Bk_{B}(f)^{k_{B}(f)} \|F_Bf\|_{L^2(-B,B)}^2 \geq \|f\|_{L^2(-1,1)}^2. 
\end{equation*}

We may consider a setting where we wish to know how much of the signal we should measure, in order to obtain a given error level. Hence we are interested in the explicit dependence of the stability constant $c_Bk_B(f)^{k_B(f)}$ on the parameter $B$. In this paper, we provide this explicit dependence, under the additional assumption that the signal $f \in H^1(-1,1)$ and $f(-1)=f(1)=0$.

When $B$ is large enough, we can  easily obtain conditions under which, the problem of recovering $f$ from $F_Bf$ is stable. This is made concrete in Theorem \ref{large_truncation}.

We will consider a function $f\in L^2(-1,1)$ to be an element of $L^2(\Re)$ by extending $f$ by zero on $\Re\setminus [-1,1]$. Similarly, we can extend a $H^1(-1,1)$ function $f$ by zero on $\Re \setminus (-1,1)$. It is clear that extension of $f$ is in $H^1(\Re)$ if and only if $f(-1) =  f(1) = 0$.

Consider the  Sobolev space 
\begin{equation*}
\begin{split}
H_0^1(-1,1) &= \{f\in H^1(\Re)\,:\, \supp(f) \subseteq (-1,1)\} \\&= \{f\in H^1(-1,1)\,:\, f(-1)=f(1)=0\}.
\end{split}
\end{equation*}

Define the frequency number of a function $ f\in H_0^1(0, 1)$ by 
\begin{equation*}
\omega(f) = \|f_x\|_{L^2(-1,1)}/\|f\|_{L^2(-1,1)}.
\end{equation*}



For functions in $H_0^1(-1,1)$, the frequency number is in fact always bounded from below. 

It is well known that the smallest eigenvalue to the Laplace problem
\begin{equation*} \label{laplace}
\begin{split}
    &-f_{xx}(x) = \lambda f(x)  \quad x\in (-1,1)\\
    &f(-1) =f(1)=0
    \end{split}
\end{equation*}
is given by
\begin{equation*}
    \lambda_{\text{min}} = \min_{\substack{f\in H_0^1[-1,1] \\ \|f\|_{L^2} = 1}} \l\{\int_{-1}^1 (f_x)^2 \d x\r\} = \min_{f\in H_0^1[-1,1]}\{\omega(f)^2\},
\end{equation*}
see \cite{evans}, chapter 6.4. Using standard methods, we find that the lowest eigenvalue for \eqref{laplace} is the smallest non-zero solution to the equation $\sin(2\sqrt{\lambda})=0$, that is $\lambda_{\t{min}} = \pi^2/4$. Hence 

\begin{equation} \label{frequencylowerbound}
\omega(f)\geq \pi/2.
\end{equation}
The concept of the frequency number has been successfully  used  recently to quantify the weak observability for evolution
systems \cite{ammari2020weak}, and  unique continuation of finite sums of eigenfunctions
of elliptic equations in divergence form \cite{osses2023improved}.


\begin{theorem}[Large Truncation] \label{large_truncation}
    Suppose $f\in H_0^1(-1,1)$. Then 
    \begin{equation} \label{general_large_truncation}
         \|f\|_{L^2(-1,1)} \leq  \l(1-\frac{\omega(f)^2}{B^2}\r)^{-1/2} \|F_Bf\|_{L^2(-B,B)} .
    \end{equation}
    In particular, if $B\geq \sqrt{\gamma}\omega(f)$ for $\gamma>1$, then
    \begin{equation} \label{particular_large_truncation}
        \|f\|_{L^2(-1,1)} \leq C_\gamma \|F_Bf\|_{L^2(-B,B)},
    \end{equation}
    with $C_\gamma = \l(1-\frac{1}{\gamma}\r)^{-1/2}$.
\end{theorem}

\begin{proof}
    Using Parseval's identity, we have that
    \begin{equation*} 
    \begin{split}
        \|f\|_{L^2[-1,1]}^2 &= \int_{-B}^B |\widehat{f}(\xi)|^2\d \xi + \int_{|\xi|\geq B} |\widehat{f}(\xi)|^2 \\
        &\leq \|F_Bf\|_{L^2(-B,B)}^2 + \frac{1}{B^2} \int_{|\xi|\geq B}  |\widehat{f}(\xi)|^2\xi^2\d \xi \\
        &\leq \|F_Bf\|_{L^2(-B,B)}^2 + \frac{\|f_x\|_{L^2(-1,1)}}{B^2} \\
        &=\|F_Bf\|_{L^2(-B,B)}^2 + \frac{\omega(f)^2}{B^2} \|f\|_{L^2(-1,1)}^2
        \end{split}
    \end{equation*}
from which we obtain \eqref{general_large_truncation}. From \eqref{general_large_truncation} we immediately obtain \eqref{particular_large_truncation}. Notice that for any $f\in H^1(-1,1)$, we needed the additional assumption that $f(-1)=f(1)=0$ in order for the following equality to hold
\begin{equation*}
    \int_{-1}^1 |f_x(x)|^2 = \int_{-\infty}^\infty \xi^2 |\widehat{f}(\xi)|^2\d \xi.
\end{equation*}
\end{proof}

The stability estimate  in Theorem \ref{large_truncation} shows that the inverse problem is well posed if the truncation parameter is large enough and covers the spatial frequency of the function. Our objective 
now is to derive a stability estimate for the inversion when the truncation parameter is smaller than the spatial frequency of the function.\\

The rest of the paper is organized as follows. In section 2, we recall the notion of harmonic measure of an open domain in the complex plane, and the Two constant Theorem. Section 3 is devoted to the derivation of a Gagliardo-Nirenberg inequality with explicit constants. The main results of the paper are provided in section 4. Theorem \ref{desired_stability} exhibits a Lipschitz stability estimate for the inverse problem with a Lipschitz constant that blows up when the truncation parameter tends to zero. Corollary \ref{eta_bound} shows that
blow up rate is of  exponential type with  explicit constants. Finally, some numerical examples of reconstruction of a compactly supported function from its noisy truncated Fourier transform are presented in section 5.

\section{Harmonic measures}
We will now introduce the notion of a harmonic measure \cite{nevanlinna1970analytic}. A real valued function, $u$ defined on $U\subset \Ce$ where $U$ is open and connected, is said to be harmonic if $u\in C^2(U)$ and $\Delta u= u_{xx} + u_{yy} \equiv 0$. Moreover, a real valued function, $u$ is said to be subharmonic on $U$, if $u$ is upper semi-continuous on $U$, that is if
\begin{equation*}
    \limsup_{z\to \xi} u(z) \leq u(\xi) \quad \text{for all} \quad \xi \in U,
\end{equation*}
and if the following inequality is satisfied for all $z\in U$
    \begin{equation} \label{mean_value_property}
            u(z) \leq \frac{1}{2\pi} \int_{0}^{2\pi} u(z+re^{i\theta}) \d \theta, \quad \text{for all} \,\,r>0\,\, \text{such that} \,\, B_r(z)\subset U.
    \end{equation}
Harmonic functions satisfy \eqref{mean_value_property} with equality.\\

The Dirichlet problem is the problem of finding a harmonic function $u$ in $U$ such that for a function $\phi: \partial U\to \mathbb{R}$ we have $\lim_{z\to \xi} u(z) = \phi(\xi)$ for all $\xi \in \partial U$. It is important to note here, that if $U$ is unbounded, then $\infty \in \partial U$, and 
the value of $\phi$ at $\infty$ should be specified. Under rather general assumptions on $U$ and $\phi$, the Dirichlet problem has a unique solution. 

\begin{theorem}[Solution to the Dirichlet problem, \cite{ransford} Corollary 4.2.6] \label{dirichlet_solution}
    Let $U\subset \Ce$ be open and simply connected, and let $\phi:\, \partial U \to \Re$ be bounded and continuous on $\partial U$ except possibly at finitely many points, $E = \{\zeta_1, \hdots \zeta_n\}$. Then there exists a unique harmonic function, $u$ on $U$ such that for all $\zeta \in \partial U \setminus E$, $\lim_{z\to \zeta} u(z) = \phi(\zeta)$. This function is given by
    \begin{equation*}
        u=\sup_{v\in \mathcal{V}} v,
    \end{equation*}
    where $\mathcal{V}$ is the family of sub-harmonic functions that satisfy
    \begin{equation*}
        \limsup_{z\to \zeta} u(z) \leq \phi(\zeta) \quad \text{for all} \quad \zeta \in \partial U
    \end{equation*}
\end{theorem}
The theorem was rigorously proved for rather arbitrary domains by David Hilbert in 1900. Today, many different proofs exists, see for example \cite{ransford}.

\begin{definition}
Let $A$ be a Borel subset of $\partial U$ with $U$ being an open and simply-connected subset of $\Ce$. The harmonic measure of $A$ in $U$ denoted 
by $w_U(\cdot, A)$, is defined as the solution to the Dirichlet problem in $U$ with the boundary data $\phi(z) = \chi_A(z)$.
\end{definition}

\begin{theorem}[Two constants theorem] \label{two_constants_theorem}
    Suppose $u: U\subset \mathbb{C}\to \mathbb{R}$ is subharmonic on $U\subset \Ce$ where $U$ is open and simply connected. Let $A$ be a Borel subset of $\partial U$. Suppose moreover that
    \begin{equation*}
        u(z)\leq M \quad \t{for all} \quad z\in U \quad  \t{and} \quad  \limsup_{z\to \xi} u(z) \leq m \quad \t{for all} \quad \xi \in A,
    \end{equation*}
    then
    \begin{equation*}
        u(z) \leq m  w_U(z,A) + M (1-w_U(z,A)).
    \end{equation*}
\end{theorem}

\begin{proof}
    Define a function $\phi: \partial U\to \mathbb{R}$ by $\phi(z) = m\chi_A(z) + M(1-\chi_A(z))$.
    
    Using Theorem \ref{dirichlet_solution} we have that $u \leq \sup_{v\in \mathcal{V}_\phi} v$, where $\mathcal{V}_{\phi}$ is the space of subharmonic functions in $U$ satisfying $\limsup_{z\in \xi} u(z) \leq \phi(\xi)$. By linearity, we have
    \begin{equation*}
    \begin{split}
        u(z) &\leq \sup_{v\in \mathcal{V}_\phi} v = m \sup_{v\in \mathcal{V}_{\chi_A}}v + M \l(1-\sup_{v\in \mathcal{V}_{\chi_A}}v\r) \\
        &= m w_U(z,A) + M (1-w_U(z,A)).
   \end{split}
    \end{equation*}
\end{proof}

\begin{theorem}\label{logf_subharmonic}
    Let $U$ be an open and simply connected subset of $\Ce$. Suppose $f$ is holomorphic on $U$. Then $\log |f|$ is subharmonic on $U$.
\end{theorem}

\begin{proof} The proof can be found in \cite{nevanlinna1970analytic}.
    We use the convention that if $z_0\in U$ is such that $|f(z_0)|=0$, then $\log |f(z_0)|=-\infty$.

    It follows easily from the Cauchy Riemann equations, that if $f$ is holomorphic in $U$, then $\Real (f)$ and $\Imag (f) $ are harmonic in $U$.
    
    Let $f(z)=r(z)e^{i\theta(z)}$ with $r=|f|$ and $\theta=\arg \,f$. We have $\log(f) = \log(r) + i\theta$, and so $\Real \log f = \log |f|$. Now suppose $z_0$ is such that $\log(f(z_0)) \neq 0$. Then there exists $\rho>0$ such that $\log(f)$ is holomorphic in $B_\rho(z_0)$. In particular
    \begin{equation*}
          \Real \log f(z_0) \leq \frac{1}{2\pi} \int_0^{2\pi}  |\Real \log f(z_0+re^{i\theta})| d\theta, \quad \for \quad 0<r<\rho.
    \end{equation*}

    Now if $f(z_0) = 0$, then clearly
    \begin{equation*}
       -\infty =  \log|f(z_0)| \leq \frac{1}{2\pi} \int_{0}^{2\pi}\log|f(z_0+re^{i\theta})|\d \theta.
    \end{equation*}
    Moreover, $\log|f(z)|$ is upper semi-continuous. Hence $\log|f|$ is subharmonic.
\end{proof}

As a consequence of Theorem \ref{two_constants_theorem} and Theorem \ref{logf_subharmonic}, we see that of $f$ is holomorphic in $U$, then we have the bound

\begin{equation*}
    |f(z)| \leq m^{w_U(z,A)}M^{1-w_U(z,A)},
\end{equation*}
where $ |f(z)| \leq M$ for all $z\in U$ and $\limsup_{z\to \xi}|f(z)|\leq m$ for all $\xi \in A$.

\section{The Gagliargo-Nirenberg inequality}

\begin{lemma}[Special case of the Gagliardo-Nirenberg inequality, \cite{nirenberg}] \label{g-n_lemma}
    Suppose that $u\in H^1(0,B)$, then 
    \begin{equation*}
        \|u\|_{L^\infty(0,B)} \leq \sqrt{2}\|u\|_{L^2[0,B]}^{1/2}\|u_x\|_{L^2(0,B)}^{1/2} + \frac{\sqrt{8}}{\sqrt{B}}\|u\|_{L^2(0,B)}.
    \end{equation*}
\end{lemma}

\begin{proof}
Since $f$
    We split the proof into three parts.
    
    \textbf{Step 1, Assuming $[0,B] = [0,1]$ and $u(0)=0$:} First, using the fundamental theorem of calculus, we have that
    \begin{equation*}
        u^2(x)= 2\int_{0}^x u(y)u_x(y)\d y,
    \end{equation*}
    hence using Hölder's inequality, we find that
    \begin{equation*}
        \|u\|_{L^\infty(0,1)}^2 \leq  2\int_{0}^1 |u(y) u_x(y)| \leq 2\|u\|_{L^2(0,1)}\|u_x\|_{L^2(0,1)}.
    \end{equation*}
    
    \textbf{Step 2, Assuming $[0,B] = [0,1]$ and $u(0)\neq 0$:}
    Define
    \begin{equation*}
        \phi(x) = \begin{cases}
            2x & 0\leq x\leq 1/2 \\
            1 & 1/2\leq x \leq 1.
        \end{cases}
    \end{equation*}
    We have that $\phi\in H^1(0,1)$. Hence $\phi u\in H^1(0,1)$, \cite{brezis}, Corollary 8.10. Moreover $\phi u(0)=0$. Hence
    \begin{equation} \label{g-n-temp}
    \begin{split}
        \sup_{1/2\leq x\leq 1} |u(x)| &\leq \|\phi u\|_{L^\infty(0,1)} \leq \sqrt{2}\|\phi u\|_{L^2(0,1)}^{1/2} \| (\phi u)_x\|_{L^2(0,1)}^{1/2} \\
        & \leq \sqrt{2}\|u\|_{L^2(0,1)}^{1/2} (\|\phi u_x\|_{L^2(0,1)}^{1/2} + \|\phi_x u\|_{L^2(0,1)}^{1/2}) \\
        & \leq \sqrt{2}\|u\|_{L^2(0,1)}^{1/2} \|u_x\|_{L^2(0,1)}^{1/2} + 2\sqrt{2}\|u\|_{L^2(0,1)}.
        \end{split}
    \end{equation}

    By symmetry, it follows that
    \begin{equation*}
         \sup_{0\leq x\leq 1/2} |u(x)| \leq \sqrt{2}\|u\|_{L^2(0,1)}^{1/2} \|u_x\|_{L^2(0,1)}^{1/2} + 2\sqrt{2}\|u\|_{L^2(0,1)}.
    \end{equation*}

    \textbf{Step 3, assuming $B>0$.}
    Suppose $u\in H^1(0,B)$. Now define
    \begin{equation*}
        \widetilde{u}(x) = u\l(Bx\r), \quad x\in [0,1]
    \end{equation*}
    We then have that $\|\widetilde{u}\|_{L^2(0,1)} = \frac{1}{\sqrt{B}} \|u\|_{L^2(0,B)}$ and $\|\widetilde{u}_x\|_{L^2(0,1)} = \sqrt{B} \|u\|_{L^2(0,B)}$. Hence
    \begin{equation*}
    \begin{split}
        \|u\|_{L^\infty(0,B)} = \|\widetilde{u}\|_{L^\infty(0,1)} &\leq \sqrt{2}\|\widetilde{u}\|_{L^2(0,1)}^{1/2} \|\widetilde{u}_x\|_{L^2(0,1)}^{1/2} + 2\sqrt{2}\|\widetilde{u}\|_{L^2(0,1)}
        \\& = \sqrt{2}\|u\|_{L^2(0,B)}^{1/2} \|u_x\|_{L^2(0,B)}^{1/2} + \frac{\sqrt{8}}{\sqrt{B}}\|u\|_{L^2(0,B)}.
        \end{split}
    \end{equation*}
\end{proof}

\section{Main result}
For a function $f\in H_0^1(-1,1)$, using the fact that $\widehat{f}$ has a holomorphic extension to all of $\Ce$, we hope to obtain a stability result 
\begin{equation} \label{desired_stability}
    \|f\|_{L^2(-1,1)} \leq k_B(\omega) \|F_Bf\|_{L^2(-B,B)}.
\end{equation}
where $k_B(\omega)$ depends explicitly on $B$ and $\omega(f)$.

\begin{theorem}[Small truncation] \label{small_truncation}
Let $f\in H_0^1(-1,1)$, and set $B_0 \geq \sqrt{\gamma} \omega(f)$ for some $\gamma>1$ to be the truncation parameter for which \eqref{particular_large_truncation} holds. Then for all $0<L$, and all $B<B_0$,
\begin{equation} \label{desired_stability}
    \|f\|_{L^2(-1,1)} \leq k_{L,B_0,B} \|F_Bf\|_{L^2(\Re)}.
\end{equation}
holds with
\begin{equation*}
     k_{L,B_0,B} = c \l(\frac{2B_0}{1-1/\gamma} \r)^{1/w} e^{2L(1-w)/w},
 \end{equation*}
 where $c = \l(\frac{B^{1/4}}{\pi^{1/4}} + \frac{2}{\sqrt{B}}\r)^2$, and where  $w=w_L(B_0,B)$ with $w_L(\cdot, B)$ is the harmonic measure of $\{0\}\times [0,B]$ in 
\begin{equation*}
    G_{L,B} = ([-L,L]\times [0,\infty)) \setminus ([0,B]\times \{0\}).
\end{equation*}
We have in particular
\begin{equation} 
\begin{split}\label{small_truncation_3_properties}
    \lim_{B\searrow 0}k_{L,B_0,B} &=\infty \\
    \lim_{B\nearrow B_0}k_{L,B_0,B} &=\l(\frac{B_0^{1/4}}{\pi^{1/4}} + \frac{2}{\sqrt{B_0}}\r)^2\frac{2B_0}{1-1/\gamma} \\
     \lim_{B_0\to \infty}k_{L,B_0,B} &= \infty.
     \end{split}
\end{equation}
\end{theorem}

\begin{corollary} \label{eta_bound}
Under the same assumptions as in Theorem \ref{small_truncation}, equation \eqref{desired_stability} holds with 

\begin{equation*}
      k_{L,B_0,B} = c \l(\frac{2B_0}{1-1/\gamma} \r)^{1/\eta} e^{2L(1-\eta)/\eta},
 \end{equation*}
 where 
\begin{equation*}
\eta = \eta_{L,B_0,B}:=\frac{2}{\pi} \arctan\l(\frac{(e^{B} -1)^{\pi/(2L)}}{\sqrt{(e^{B_0}-1)^{\pi/L}-(e^{B}-1)^{\pi/L}}}\r), \quad B<B_0.
\end{equation*}
\end{corollary}

The stability estimate in Theorem \ref{small_truncation} and  the bounds in Corollary \ref{eta_bound} 
show that the inverse problem is well posed if the truncation parameter is large enough and covers the spatial 
frequency of the function, and it becomes exponentially unstable when the truncation parameter shrinks to zero.  
This is in agreement  with the fact that  the singular
values of $F_B$ decay exponentially, and are increasing with respect to $B$ \cite{osipov, bao2010multi}.

\begin{proof}[Proof of Theorem \ref{small_truncation}]
    Set $B_0 = \sqrt{\gamma}\omega(f)$.  From Theorem \ref{large_truncation}, it follows that $\|\widehat{f}\|_{L^2(-B_0,B_0)}^2\geq (1-1/\gamma)\|f\|^2_{L^2(-1,1)}$.

    \sloppy Assume without loss of generality, that $\sup_{\xi \in [0,B_0)}|\widehat{f}(\xi)| =\sup_{\xi \in (-B_0,B_0)}|\widehat{f}(\xi)|$. We then have 
    \begin{equation} \label{L2_Linfty}
    \begin{split}
        \sup_{\xi \in [0,B_0)}|\widehat{f}(\xi)| &=\sup_{\xi \in (-B_0,B_0)}|\widehat{f}(\xi)| \\&\geq \frac{\|\widehat{f}\|_{L^2(-B_0,B_0)}}{\sqrt{2B_0}} \\&\geq \frac{1}{\sqrt{2B_0}}\sqrt{1-\frac{1}{\gamma}}\|f\|_{L^2(-1,1)}.
        \end{split}
    \end{equation}

    Now consider the half strip $S_L = \{z \in \Ce\,:\, \text{Re}(z)>0, \, |\text{Im}(z)|<L\}$, and the deleted half strip
    \begin{equation*}
        G_{L,B}:=S_L\setminus ([0,B]\times \{0\}).
    \end{equation*}
    Let $M$ be a constant such that $M\geq \limsup_{z\in G_{L,B}} |\widehat{f}(z)|$.
    For $z\in G_{L,B}$ we have
    \begin{equation*}
        |\widehat{f}(z)| \leq \frac{1}{\sqrt{2\pi }} \l|\int_{-1}^1 f(x)e^{ixz} \d x \r| \leq  \frac{1}{\sqrt{2\pi}}\|f\|_{L^1(-1,1)} e^{L} \leq \frac{1}{\sqrt{\pi}} \|f\|_{L^2(-1,1)} e^{L},
    \end{equation*}
    where the last inequality follows from Hölder's inequality.
    Hence we can choose $M = \frac{1}{\sqrt{\pi}}\|f\|_{L^2(-1,1)} e^{L}$.
    Now let $m:=\sup_{\xi \in [0,B)} |\widehat{f}(\xi)|$, and let $w_L(z,B)$ be the Harmonic measure of $[0,B] \times \{0\}$ in $G_{L,B}$. We then can apply Theorem \ref{two_constants_theorem} to obtain 
    \begin{equation*}
        |\widehat{f}(z)|\leq m^{w_{L}(z,B)}M^{1-w_{L}(z,B)}, \quad z \in G_{L,B}.
    \end{equation*}

    For $z\in [0,\infty)$, $z\mapsto w_L(z,B)$ is monotonically decreasing. Hence $\inf_{z\in [0,B_0)} w_L(z,B) = w_L(B_0,B)$. Moreover, the function $x\mapsto m^xM^{1-x}$ is monotonically decreasing for $x\in (0,1)$. Combining this with \eqref{L2_Linfty}, and writing $w=w_L(B_0,B)$, we have
    \begin{equation} \label{bound_later_using_eta}
    \begin{split}
        \frac{1}{\sqrt{2B_0}}\sqrt{1-\frac{1}{\gamma}}\|f\|_{L^2(-1,1)} &\leq \sup_{z\in [0,B_0)}|\widehat{f}(z)|\leq m^{w}M^{1-w}.
        \end{split}
    \end{equation}

    Using the Gagliardo-Nirenberg inequality \ref{g-n_lemma} and that 
    \begin{equation*}
    \begin{split}
    \|(\widehat{f})_x\|_{L^2(0,B)}^2 &= \int_0^B \l| \frac{d}{d\xi} \int_{-1}^1 f(x) \e^{\i x\xi} \d x \frac{1}{\sqrt{2\pi}}\r|^2\d \xi 
    \\&= \frac{1}{2\pi} \int_0^B \l| \int_{-1}^1 \i x f(x) \e^{\i x\xi}\d x\r|^2
    \\&\leq \frac{1}{2\pi}\int_0^{B} \l( \int_{-1}^1 |f(x)|\d x\r)^2 \d \xi
    \\&\leq \frac{B}{\pi} \|f\|^2_{L^2(-1,1)},
    \end{split}
    \end{equation*}
    we find that
    \begin{equation*}
    \begin{split}
        m &= \sup_{\xi \in [0,B)} |\widehat{f}(\xi)| \leq \sqrt{2}\|(\widehat{f})_x\|^{1/2}_{L^2(0,B)} \|\widehat{f}\|_{L^2(0,B)}^{1/2} + \frac{\sqrt{8}}{\sqrt{B}}\|\widehat{f}\|_{L^2(0,B)}
        \\& \leq \l(\frac{\sqrt{2}B^{1/4}}{\pi^{1/4}} + \frac{\sqrt{8}}{\sqrt{B}}\r)\|f\|^{1/2}_{L^2(-1,1)} \|\widehat{f}\|_{L^2(0,B)}^{1/2} \\& = \sqrt{2c} \|f\|^{1/2}_{L^2(-1,1)} \|\widehat{f}\|_{L^2(0,B)}^{1/2}.
        \end{split}
    \end{equation*}
    with $c =  \l(\frac{B^{1/4}}{\pi^{1/4}} + \frac{2}{\sqrt{B}}\r)^2$. Thus continuing \eqref{bound_later_using_eta}
\begin{equation}
\begin{split}
    \frac{1}{\sqrt{2B_0}}&\sqrt{1-\frac{1}{\gamma}}\|f\|_{L^2(-1,1)} \\&\leq (\sqrt{2c} \|f\|^{1/2}_{L^2(-1,1)} \|\widehat{f}\|_{L^2(0,B)}^{1/2} )^{w}(\sqrt{2}\|f\|_{L^2(-1,1)}e^{L})^{1-w}
    \\& \leq (c)^{w/2} \|\widehat{f}\|_{L^2(0,B)}^{w/2}  \|f\|_{L^2(-1,1)}^{1-w/2 }  e^{ L(1-w)}.
    \end{split}
\end{equation}
We can now isolate $\|f\|_{L^2(-1,1)}$ to obtain
\begin{equation*}
\begin{split}
    \|f\|_{L^2(-1,1)} &\leq c \l(\frac{2B_0}{1-1/\gamma} \r)^{1/w} e^{2L(1-w)/w}  \| \widehat{f}\|_{L^2(0,B)}
    \\ & \leq c \l(\frac{2B_0}{1-1/\gamma} \r)^{1/w} e^{2L(1-w)/w}  \| F_B\|_{L^2(-B,B)}.
    \end{split}
\end{equation*}

The three properties \eqref{small_truncation_3_properties} follow from following facts, namely that for $B=B_0$, $w = w_L(B_0,B_0)=1$. Moreover, for fixed $B_0>0$, $\lim_{B\to 0} w_L(B_0,B)=0$ and for fixed $B>0$, $\lim_{B_0\to \infty}w_L(B,B_0)=0$.

Moreover, since $B_0=\sqrt{\gamma}\omega(f)\geq  \sqrt{\gamma}\pi/2$, we are guaranteed that $\lim_{w\to 0}\l(\frac{2B_0}{1-1/\gamma}\r)^{1/w}=\infty$.

This proves Theorem \ref{small_truncation}.
    
\end{proof}

\begin{proof}[Proof of Corollary \ref{eta_bound}]

    The function $x\mapsto m^xM^{1-x}$ is monotonically decreasing for $0<x<1$, whenever $m<M$. In \cite{bao} it was proved that $\eta_{L,B_0,B}\leq w_L(B_0,B)$ whenever $0<L<\pi/2$.
    
    Thus using equation \eqref{bound_later_using_eta} in the proof of theorem \ref{small_truncation}, we have
    \begin{equation*}
        \frac{1}{\sqrt{2B_0}}\sqrt{1-\frac{1}{\gamma}}\|f\|_{L^2(-1,1)} \leq m^{\eta_{L,B_0,B}}M^{1-\eta_{L,B_0,B}}.
    \end{equation*}
    The rest of the proof follows that of theorem \ref{small_truncation}, with $\eta = \eta_{L,B_0,B}$ instead of $w$.
\end{proof}

\section{Numerical example}
Let $f\in H_0^1(-1,1)$ and recall that the truncated Fourier transform of $f$ is given by
\begin{equation*}
    F_Bf(\xi) =\frac{1}{\sqrt{2\pi}} \int_{-1}^1 f(x) e^{\i x\xi}\d x, \quad \xi\in [-B.B]
\end{equation*}
Given a noisy measurement $g(\xi) = \frac{1}{\sqrt{2\pi}}F_Bf(\xi) + \varepsilon(\xi)$ with $\varepsilon$ being a random field with zero mean, we can approximate $f$ by applying the inverse Fourier transform to $g$. In the case where $\varepsilon \equiv 0$, we have that
\begin{equation*}
    f^{\rm rec}(x) :=(F^{-1}g)(x) = (\phi * f)(x),
\end{equation*}
where $\phi(x) = \sin(Bx)/(\pi x)$.

\subsection{Numerical implementation}

Consider the ground truth signal $f\in H_0^1[-1,1]$.

The measurements that we will consider in this example is noisy samples of $Ff = \widehat{f}$ taken at a uniform grid
\begin{equation*}
    \Xi_{B,h} = (-B,\, -B+h,\, -B+2h,\hdots,\, B-h) = (\xi_1,\hdots, \xi_M),
\end{equation*}
where we impose that $M = 2B/h \in \mathbb{N}$.

We then generate one measurement $g = \{g_m\}_{m=1}^M$, with
\begin{equation} \label{generate measurement}
    g_m = Ff(\xi_m) + \delta (\varepsilon_m^r + \i \varepsilon_m^{i}),
\end{equation}
where $\varepsilon_m^r, \varepsilon_m^{i}$, $m=1,\hdots, M$ is independent and identically distributed Gaussian noise with zero mean and unit variance.

To approximate the inverse Fourier transform, we will use the Fractional Fourier Transform (FRFT), \cite{bailey}, which uses the Fast Fourier transform to calculate the integral in the inverse Fourier transform using a quadrature rule.

Using the the FRFT on the measurements, $g$, we obtain a reconstruction
\begin{equation} \label{discrete reconstruction}
    f^{\rm rec} = (f_1^{\rm rec}, \hdots , f_M^{\rm rec}) = \text{FRFT}(g),
\end{equation}
where $f_m^{\rm rec}$ approximates $f(x_m)$ with $x_m = -1 + (m-1)2/M$. We can regard $f^{\rm rec}$ as an element of $L^2(-1,1)$ by doing interpolation of \eqref{discrete reconstruction} at the nodes $(x_m)_{m=1}^M$. We can approximate the $L^2$-reconstruction error of $f^{\rm rec}$ by computing
\begin{equation*}
    E^{\rm rec}(g) = \l[ \frac{2}{M} \sum_{m=1}^M  |f_m^{\rm rec} - f(x_m)|^2 \r]^{1/2}.
\end{equation*}
Specifically, $E^{\rm rec}(g)$ is the left-point quadrature rule approximation to the $L^2$-norm of the error.

In this numerical example, we consider reconstructing the eigenfunctions of the Laplacian, 
\begin{equation*}
    f_k(x) = \sin\l( \frac{k\pi (x+1)}{2} \r), \quad x\in [-1,1],
\end{equation*}
where $k\in \mathbb{N}$.

Their Fourier transform is given by 
\begin{equation} \label{laplacian fourier}
    Ff_k(\xi) = \frac{-{ e}^{\i \xi  } \left(-1\right)^k \pi  k +\pi  k \,e^{\mathrm{-\i} \xi}}{\pi  \left(k^{2} \pi^{2}-4 \xi^{2}\right)}, \quad \xi \in \Re.
\end{equation}
Using the expression \eqref{laplacian fourier}, we generate the measurements, $g$ as in \eqref{generate measurement}. We expect the error $E^{\rm rec}(g)$ to depend on the following parameters
\begin{itemize}
    \item $k$: We will expect the reconstruction error to be higher for larger values of $k$, as we have the relationship
    \begin{equation*}
        \omega(f_k) = \|(f_k)_x\|_{L^2(-1,1)}/\|f_k\|_{L^2(-1,1)}= k\pi/2.
    \end{equation*}

    \item $B$: We expect the reconstruction error to be low, if the truncation bandwidth $B$ is high.

    \item $\delta$: We expect the reconstruction error to be high, if the noise level $\delta$ is high. In particular, we expect that even for a very large truncation bandwidth $B$, the reconstruction error will be around $\delta$.
    \item $M$: The number of samples. In this example, we set the sampling rate, $r$, to be constant, that is we set $M = \lceil r B \rceil$.
\end{itemize}

\begin{figure}[h]
\centering     
\subfigure[$k=4$]{\label{fig:a}\includegraphics[width=60mm]{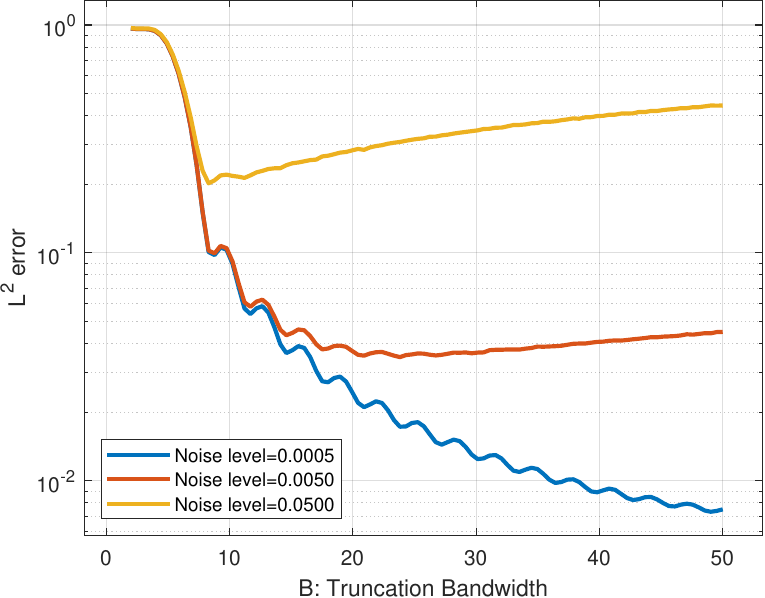}}
\subfigure[$k=15$]{\label{fig:b}\includegraphics[width=60mm]{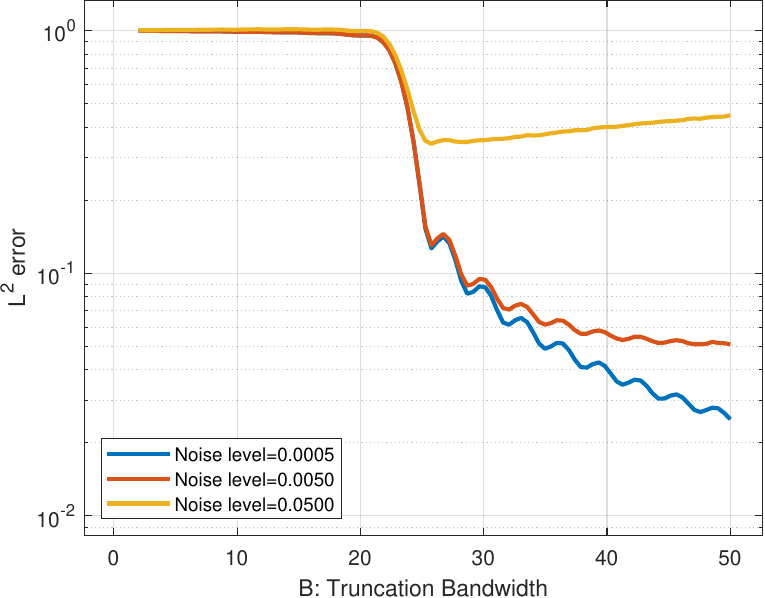}}
\caption{Reconstruction error using measurements taking on a grid $\Xi_{B,h}$. For each value of the noise level $\delta$ and for each value of $B$, the reconstruction error is computed from 1000 simulations. The plots shows the average of these simulations.}
\label{fig:error_vs_bw}
\end{figure}

From the experiments, we observe in Figure \ref{fig:error_vs_bw} that for a low bandwidth $B$, the reconstruction error is high. The error decreases dramatically at a "critical" bandwidth around $B=\omega(f_k) = \frac{\pi}{2}k$. This critical bandwidth does not seem to depend on the noise level, so in the following, we will set the noise level to zero. Notice also, that after the critical bandwidth, the reconstruction error seems to increase. This is especially apparent for $k=4$ and the highest noise level.\\

We define the critical bandwidth by
\begin{equation*}
    B_0(e_{\rm cut},k) = \min \{B_0\,:\,E^{\rm rec}(F_{B_0}f_k)>e_{\rm cut}\}.
\end{equation*}

The critical bandwidth for a range of $k$ and $e_{\rm cut}$ is shown in Figure \ref{fig:critical_bandwidth}. The critical bandwidth depends on $e_{\rm cut}$ and $k$ in the following way $B_0(e_{\rm cut}, k) = Ck + C_{e_{\rm cut}}$. Using a least squares fit on the points in Figure \ref{fig:critical_bandwidth}, we find that $C = 1.58 \approx \pi/2$, and where the offset $C_{0.2}>C_{0.5}>C_{0.7}$ is a small number. This verifies Theorem \ref{large_truncation}, namely that for $B>(1+\epsilon)\omega(f)$, with $\epsilon>0$, we can obtain stable reconstructions, as $\|F_Bf\|_{L^2(-B,B)}\geq C_\epsilon \|f\|_{L^2(-1,1)}$, where $C_\epsilon>0$. For $B<\omega(f)$ we also have a stability estimate as seen in Theorem \ref{small_truncation}, but here the stability constant increases very rapidly as $\omega(f)-B$ increases. \\

In the 1963 paper by H. J. Landau  \cite{landau}, it was found that the $\lfloor B\rfloor - 1$'th singular value of $F_B$, $\sigma_{\lfloor B\rfloor -1}$ satisfies $\sigma_{\lfloor B\rfloor -1}>0.4$, whereas $\sigma_{\lfloor B\rfloor + 1}<0.6$. In general, the first $\lfloor B\rfloor $ singular values are close to 1, whereas $\sigma_{\lfloor B\rfloor +k}$ decays at least exponentially with $k$.

\begin{figure}[h]
\centering     
\subfigure[Noise $\delta = 0$]{\label{fig:a}\includegraphics[width=60mm]{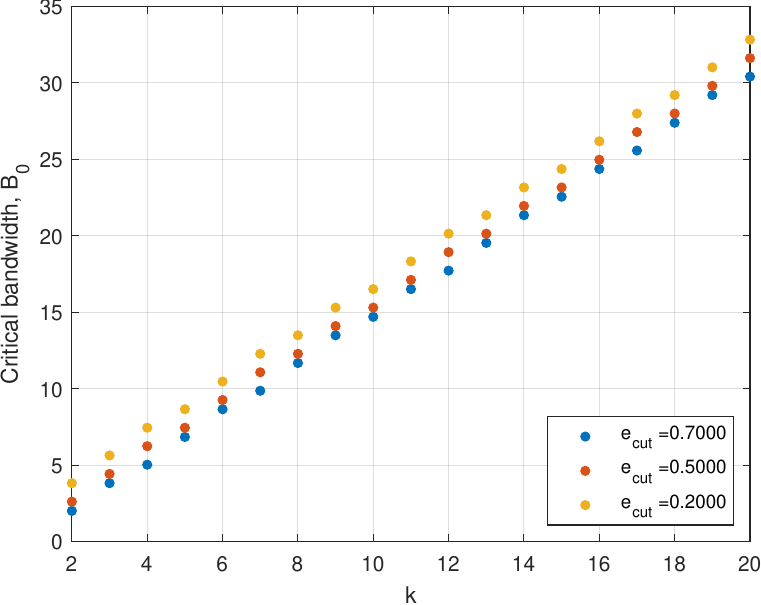}}
\subfigure[Noise $\delta = 0.05$]{\label{fig:b}\includegraphics[width=60mm]{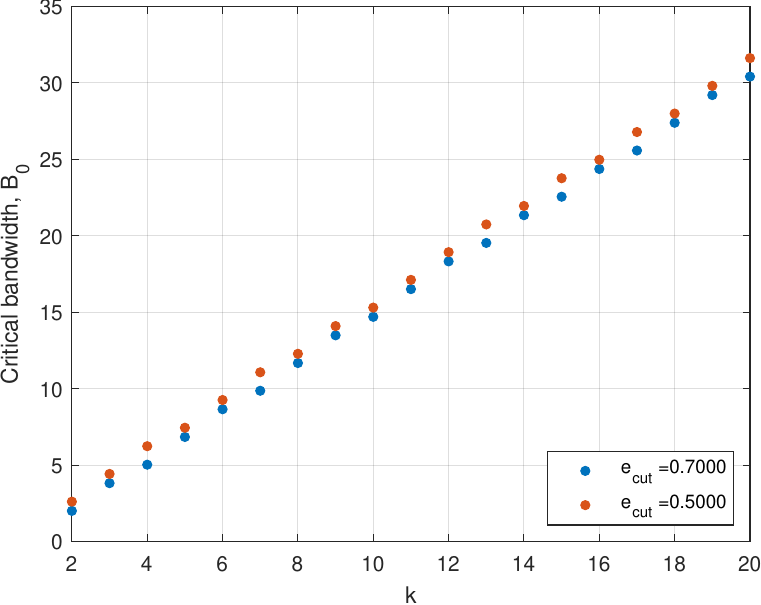}}
\caption{Critical bandwidth, $B_0$ versus the index $k$ for the eigenfunctions $f_k$. The plot shows different cutoff errors $e_{\rm cut}$. With $e_{\rm cut } = 0.2$ and $\delta =0.05$, the critical bandwidth is $B_0 = \infty$.}
\label{fig:critical_bandwidth}
\end{figure}

\subsection{Error analysis for the numerical example}

Recall that the reconstruction error of the noisy measurement $g = F_Bf +  \epsilon \in \Ce^{M}$ is given by
\begin{equation*}
    E^{\rm rec}(g) = \l[ \frac{2}{M} \sum_{m=1}^M  |f_m^{\rm rec} - f(x_m)|^2 \r]^{1/2}.
\end{equation*}
Here the spacing in-between samples, $h$, and the total number of samples satisfy $M =  2 h B\in \Ne$. The sampling rate is refined as $r_h = 1/h$. The error $E^{\rm rec}$ can be decomposed into several parts:

\begin{equation*}
    E_{B} = \|f - F^{-1}F_Bf\|_{L^2(-1,1)}.
\end{equation*}
This is the truncation error. Theorem \ref{large_truncation} and \ref{small_truncation} give upper bounds on this error. Then there is the sampling error
\begin{equation*}
    E_{h} = \l\| F^{-1}F_Bf -  \Tilde{F}^{-1}F_Bf \r\|_{L^2(-1,1)}
\end{equation*}
with
\begin{equation*}
    \widetilde{F}(g) = \sum_{m=0}^{M-1} g\l(\frac{2m B}{M} - B\r) e^{2\pi i x(2mB/M - B)},
\end{equation*}
which for a fixed $x$ is the standard left-point numerical quadrature rule for computing the integral $\tilde{F}^{-1} F_Bf(x)$.

For an arbitrary function $f\in H^1(-1,1)$, define
\begin{equation*}
    E_I(f) = \l| \|f\|_{L^2(-1,1)} - \l(\sum_{m=1}^M |f(x_m)|^2\r)^{1/2} \r|.
\end{equation*}
This error depends on the number of samples and the regularity of $f$.\\

For simplicity, let $I_h(f^{\rm rec})$ be an interpolation of the reconstruction samples: $I_h: \Ce^n \to L^2(-1,1)$. For now it does not matter what this interpolation is, we can consider 1st order spline, so that $I_hf^{\rm rec}$ is piece-wise affine.

Now let's consider how we might decompose the reconstruction error:
\begin{equation*}
\begin{split}
    E(g) &= E_I(f-I_hf^{\rm rec}) +  \|f-I_h f^{\rm rec}\|_{L^2(-1,1)}
    \\&\leq E_I(f-I_hf^{\rm rec}) + E_B + E_h + \|\widetilde{F}^{-1} (g) - \widetilde{F}^{-1}(\varepsilon) - I_hf^{\rm rec}\|
    \end{split}
\end{equation*}

Now choose the interpolation scheme $I_h$ such that
\begin{equation*}
    I_h f^{\rm rec} = I_h (\{(\widetilde{F}^{-1} g)(x_m)\}_{m=1}^{M}) = \widetilde{F}^{-1} g.
\end{equation*}

We then obtain
\begin{equation*}
    E(g)\leq E_I(f-\widetilde{F}^{-1}g) + E_B + E_h + E_\varepsilon,
\end{equation*}

where $E_\varepsilon = \| \widetilde{F}^{-1}(\varepsilon) \|_{L^2(-1,1)}$.

We can consider what happens in the limit $h=0$. Then we have $0=E_I(f-\widetilde{F}^{-1}g)=E_h$. Now suppose $\varepsilon \in L^\infty (-B,B)$, then we have
\begin{equation*}
\begin{split}
    E(g) &\leq  E_B + E_\varepsilon
    \\&= E_{B} + \l(\int_{-1}^1 \l|\int_{-B}^B \varepsilon(\xi) e^{-i x\xi}\d   \xi \r|^2 \d x\r)^{1/2}
    \\&\leq E_B  + \sqrt{8}B\|\varepsilon \|_{L^\infty(-B,B)}
    \\&\leq  E_B  + 2\sqrt{B} \|\varepsilon \|_{L^2(-B,B)}
    \\&\leq  E_B  + \sqrt{2}\|\varepsilon \|_{L^1(-B,B)}
    \end{split}
\end{equation*}

In the discrete case $h>0$, we indeed also see that the error $E(g)=\mathcal{O}(\sqrt{B})$ as $B\to \infty$, see Figure \ref{fig:noise-error}.

\begin{figure}[h]
    \centering
    \includegraphics[width=7cm]{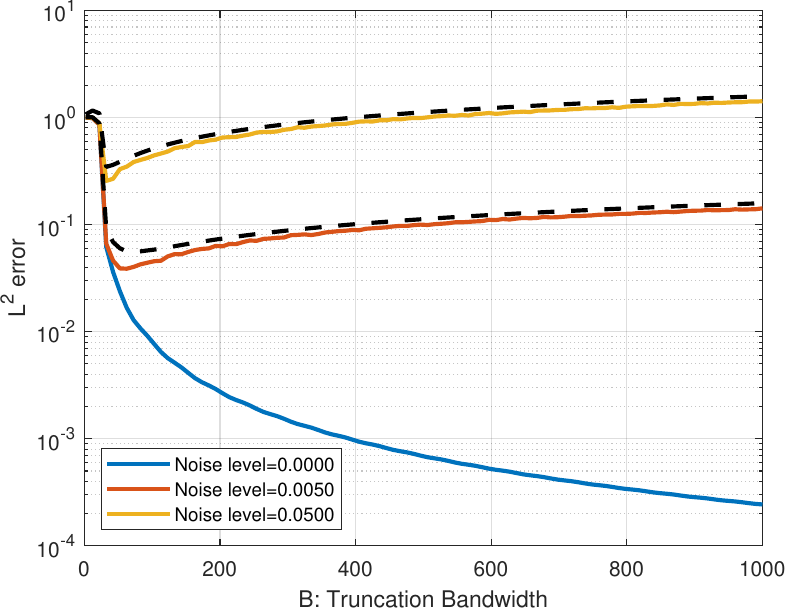}
    \caption{Reconstruction error using noisy measurements (see \eqref{generate measurement}) of $\widehat{f}_{15}$ taking on a grid $\Xi_{B,h}$. The black dottet lines show the sum of the error in the noise free case (blue line) and $\delta \sqrt{B}$.}
    \label{fig:noise-error}
\end{figure}

\appendix

\bibliographystyle{abbrv}
\bibliography{stability_truncated_fourier_transform} 
\end{document}